\documentclass[12pt]{article}
\usepackage[a4paper,margin=0.9in]{geometry}

\usepackage[colorlinks,citecolor=magenta,linkcolor=black]{hyperref}
\pdfpagewidth=\paperwidth \pdfpageheight=\paperheight
\usepackage{amsfonts,amssymb,amsthm,amsmath,eucal,tabu,url}
\usepackage{pgf}
 \usepackage{array}
 \usepackage{tikz-cd}
 \usepackage{pstricks}
 \usepackage{pstricks-add}
 \usepackage{pgf,tikz}
 \usetikzlibrary{automata}
 \usetikzlibrary{arrows}
 \usepackage{indentfirst}
\usepackage{tabularx} 

\theoremstyle{plain}
\newtheorem{thm}{Theorem}
\newtheorem{theorem}[thm]{Theorem}
\newtheorem*{theoremA}{Theorem A}
\newtheorem*{theoremB}{Theorem B}
\newtheorem*{theoremC}{Theorem C}
\newtheorem*{theoremD}{Theorem D}

\newtheorem{conjecture}[thm]{Conjecture}

\theoremstyle{definition}
\newtheorem{definition}[thm]{Definition}
\newtheorem{remark}[thm]{Remark}
\newtheorem{example}[thm]{Example}

\newtheorem{thevarthm}[thm]{\varthmname}

\newenvironment{varthm*}[1]{\trivlist\item[]{\bf #1.}\it}{\endtrivlist}

\def\keywordname{{\bfseries Keywords}}%
\def\keywords#1{\par\addvspace\medskipamount{\rightskip=0pt plus1cm
\def\and{\ifhmode\unskip\nobreak\fi\ $\cdot$
}\noindent\keywordname\enspace\ignorespaces#1\par}}
\def\subclassname{{\bfseries Mathematics Subject Classification
(2020)}\enspace}
\def\subclass#1{\par\addvspace\medskipamount{\rightskip=0pt plus1cm
\def\and{\ifhmode\unskip\nobreak\fi\ $\cdot$
}\noindent\subclassname\ignorespaces#1\par}}

\begin{document}
\title{Defect of irreducible plane curves with simple singularities}
\author{Piotr Pokora}
\date{\today}
\maketitle
{\centering\footnotesize To Professor Arkadiusz P\l oski, in memoriam. \par}
\thispagestyle{empty}
\begin{abstract}
In this note we focus on the defect of singular plane curve that was recently introduced by Dimca. Roughly speaking, the defect of a reduced plane curve measures the discrepancy from the property of being a free curve. We find some lower-bound on the defect for certain classes of irreducible plane curves admitting nodes, ordinary cusps and ordinary triple points. The main result of the note tells us that reduced simply singular plane curves with sufficiently high Arnold exponents are never free.
\subclass{14N25, 14H50, 32S25}
\end{abstract}

In this note we study the defect for some classes of irreducible plane curves. This notion has been recently introduced by Dimca in \cite{Dimca1} and we have many extremely interesting questions revolving around this notion that we may want to study. Probably the most interesting conjecture devoted to the defect is \cite[Conjecture 3.7]{SurDim}, which tells us that the defect for line arrangements in $\mathbb{P}^{2}_{\mathbb{C}}$ is combinatorially determined, and this should be understood as a broad generalization of Terao's freeness conjecture. In order to present further questions, we need a solid preparation based on \cite{Dimca}.

Let $S := \mathbb{C}[x,y,z] = \bigoplus_{k}S_{k}$ be the graded polynomial. In the note we consider reduced and not necessarily irreducible curves $C \subset \mathbb{P}^{2}_{\mathbb{C}}$. We denote by $\partial_{x}, \partial_{y}, \partial_{z}$ the partial derivatives and we define ${\rm Der}(S) = \{ \partial := a\cdot \partial_{x} + b\cdot \partial_{y} + c\cdot \partial_{z}, \,\, a,b,c \in S\}$ which is the free $S$-module of $\mathbb{C}$-linear derivations of the ring $S$. Now for a reduced curve $C \, = \, \{ f=0 \}$ with  $f \in S_{d}$ being homogeneous, we define
$${\rm D}(f) = \{ \partial \in {\rm Der}(S) \, : \, \partial (f) \in \langle f \rangle \}.$$
It means that ${\rm D}(f)$ is the graded $S$-module of derivations preserving the ideal $\langle f \rangle$. Recall that for a reduced curve $C \, = \, \{ f=0 \}$ in $\mathbb{P}^{2}_{\mathbb{C}}$ we have the following decomposition \cite[pp. 151-152]{Dimca}:
$${\rm D}(f) = {\rm D}_{0}(f) \oplus S\cdot \delta_{E},$$
where $\delta_{E} = x\partial_{x} + y\partial_{y} + z\partial_{z}$ is the Euler derivation, and 
$${\rm D}_{0}(f) = \{ \partial \in {\rm Der}(S) \, : \, \partial f = 0\},$$
i.e., the set of all $\mathbb{C}$-linear derivations of $S$ killing the polynomial $f$. It is classically known, see for example \cite[p. 151]{Dimca}, that ${\rm D}_{0}(f)$ can be identified with the $S$-module of all non-trivial Jacobian relations for the partials of $f$, namely
$${\rm AR}(f) = \{ (a,b,c)\in S^{3} \, : \, a\cdot \partial_{x}\, f + b\cdot \partial_{y} \, f + c\cdot \partial_{z} \, f =0 \}.$$
We have some numerical invariants that one can associated with a curve $C \, = \, \{ f=0 \}$ in $\mathbb{P}^{2}_{\mathbb{C}}$, one of them is the minimal degree among derivations killing $f$, i.e., 
$${\rm mdr}(f) = {\rm min}\{r \in \mathbb{N} \, : \, {\rm D}_{0}(f)_{r} \neq 0 \} = {\rm min}\{r \in \mathbb{N} \, : \, {\rm AR}(f)_{r}\neq 0\}.$$
Sometimes we will write ${\rm mdr}(C)$ for a given curve $C \subset \mathbb{P}^{2}_{\mathbb{C}}$.

For a homogeneous polynomial $f \in S$ of degree $d$ we define its Jacobian ideal $J_{f} :=\langle \partial_{x}\, f , \partial_{y}\, f, \partial_{z}\, f \rangle$. Now we define by $I_{f}$ the saturation of $J_{f}$ with respect to the irrelevant ideal $\mathfrak{m} = \langle x,y,z \rangle$ as $I_{f} := \bigcup_{k \geq 0}(J_{f} : \mathfrak{m}^{k})$. The Jacobian module of $f$ is defined as
$$N(f) = I_{f} / J_{f}.$$
The Jacobian module provides information about the curve that is defined by $f \in S$. In order to show its strength, let us present the following definition, see \cite[Definition 8.1]{Dimca} or \cite[Remark 4.7]{DimcaSernesi}.
\begin{definition}
A reduced curve $C \, = \, \{ f=0 \}$ in $\mathbb{P}^{2}_{\mathbb{C}}$ defined by a homogeneous polynomial $f \in S$ is \textbf{free} if ${\rm D}(f)$, or equivalently ${\rm D}_{0}(f)$, is a free graded $S$-module. 
\end{definition}
\noindent
It turns out that the freeness of $C \, = \, \{ f=0 \}$ is equivalent to the condition $N(f)=0$, i.e., the Jacobian ideal is saturated, for details see \cite[Proposition 1.9]{ST}. 

For a reduced curve $C \, = \, \{ f=0 \}$ in $\mathbb{P}^{2}_{\mathbb{C}}$ we set 
$$n(f)_{j} = {\rm dim} \, N(f)_{j},$$ and we define the following invariant
$$\nu(C) = {\rm max}\{n(f)_{j}\}_{j}.$$
The invariant $\nu(C)$ is called the \textbf{defect}, or the \textbf{defect from the freeness property}. It is very difficult to compute the defect of a given curve $C$ by using the above definition. However, Dimca showed the following crucial result. Before we formulate it, let us fix the notation. If $C=\{f(x,y)=0\}$ is a germ of an isolated plane curve singularity at $p=(0,0)$, then we define its local Tjurina number as follows 
$$\tau_{p}(C)=\dim_\mathbb{C}\left(\mathbb{C}[x,y] /\bigg\langle f,\frac{\partial f}{\partial x},\frac{\partial f}{\partial y}\bigg\rangle \right).$$  
Now for a reduced curve $C \, = \, \{ f=0 \}$ in $\mathbb{P}^{2}_{\mathbb{C}}$ we define its total Tjurina number as
$$\tau(C) = \sum_{p \in {\rm Sing}(C)} \tau_{p}(C),$$
where ${\rm Sing}(C)$ denotes the set of all singular points of $C$.
\begin{theorem}[{\cite[Theorem 1.2]{Dimca1}}]
\label{deff}
Let $C \, = \, \{ f=0 \}$ be a reduced plane curve of degree $d$ and $r= {\rm mdr}(C)$. Then the following hold.
\begin{itemize}
    \item If $r < (d-1)/2$, then $\nu(C) = (d-1)^{2}-r(d-1-r)-\tau(C)$.
    \item If $r \geq (d-2)/2$, then
    $$\nu(C) = \bigg\lceil\frac{3}{4}(d-1)^{2}\bigg\rceil - \tau(C).$$
\end{itemize}
\end{theorem}

There are many interesting and difficult open problems regarding the notion of the defect and here we would like to recall two the most important conjectures. The first, mentioned at the very beginning of the introduction, can be seen as a vast generalization of Terao's freeness conjecture and is devoted to line arrangements.
\begin{conjecture}
For a line arrangement $\mathcal{L} \subset \mathbb{P}^{2}_{\mathbb{C}}$ the defect $\nu(\mathcal{L})$ is determined by the intersection lattice of $\mathcal{L}$. More precisely, if $\mathcal{L}_{1}$ and $\mathcal{L}_{2}$ are two line arrangements that have isomorphic intersection lattices, then $\nu(\mathcal{L}_{1}) = \nu(\mathcal{L}_{2})$.
\end{conjecture}
This conjecture seems to be extremely difficult and for more details about it we refer the reader to an excellent recent survey by Dimca \cite{SurDim}. In the case of our note, we focus on the case of irreducible plane curves, and in order to present the main motivation for our research we need two additional definitions.
\begin{definition}
 A \textbf{plane rational cuspidal curve} is a rational curve $C \subset \mathbb{P}^{2}_{\mathbb{C}}$ having only unibranch singularities.   
\end{definition}
It is also necessary to introduce another important class of curves that was defined in \cite{DimStic}.
\begin{definition}
A reduced curve $C \subset \mathbb{P}^{2}_{\mathbb{C}}$ is \textbf{nearly free} if $\nu(C) = 1$.
\end{definition}
In the light of the above definitions, we have the following truly surprising conjecture.
\begin{conjecture}
Any rational cuspidal curve $C$ is either free or nearly free.
\end{conjecture}
In the present note, strongly motivated by the above conjecture, we want to continue the idea of studying the defect for some natural classes of irreducible plane curves, since, apart from the above conjecture, \textbf{we do not have any general prediction or results devoted to such curves}.

Our first result is devoted to nodal curves.
\begin{definition}
We say that an irreducible and reduced curve $C_{d} \subset \mathbb{P}^{2}_{\mathbb{C}}$ of degree $d$ is \textbf{nodal} if every singular point of $C_{d}$ is an ordinary double point, i.e., a singular point having the local normal form $x^{2} + y^{2}=0$.
\end{definition}
\begin{remark}
We will refer to ordinary double points as nodes. Furthermore, if $n_{2}(C_{d})$ denotes the number of nodes of an irreducible and reduced plane curve $C_{d} \subset \mathbb{P}^{2}_{\mathbb{C}}$ of degree $d \geq 3$, then by the genus formula we have $n_{2}(C_{d}) \leq \frac{(d-1)(d-2)}{2}$.
\end{remark}
\begin{theoremA} Let $C_{d}$ be a nodal plane curve of degree $d\geq 4$. Then
$$\nu(C_{d}) \geq \frac{1}{4}(d^{2}-1).$$
In particular, the defect for nodal curves can be arbitrarily large.
\end{theoremA}
The next result is devoted to irreducible and reduced plane curves of genus zero admitting only nodes and ordinary triple points as singularities.
\begin{theoremB}
There exists an irreducible and reduced plane curve $K_{3k}$ of degree $d=3k$ with $k\geq 3$ of genus zero that admits exactly $2k$ ordinary triple points and nodes as singularities such that
$$\nu(K_{3k}) \geq \frac{1}{4}(9k+1)(k-1).$$
\end{theoremB}
\noindent
Here by an ordinary triple point we mean a singularity defined by the local normal form $y^{2}x + x^{3}=0$.

Finally, we focus on certain cuspidal curves that were constructed by Ivinskis \cite{Ivinskis}.
\begin{theoremC}
There exists an irreducible and reduced plane curve $C_{6k}$ of degree $d=6k$ with $k\geq 1$ that admits exactly $9k^2$ ordinary cusps and no other singularities such that
$$\nu(C_{6k}) = 9k^2-9k+1 = g(C_{6k}),$$
where $g(C_{6k})$ denotes the genus of $C_{6k}$.
In particular, $C_{6}$ is nearly free.
\end{theoremC}
\noindent
For completeness, recall that an ordinary cusp is a singularity defined by the local normal form $y^2 + x^{3} = 0$.

Our results show that rational cuspidal plane curves are very special and it allows us to justify the heuristic phenomenon that it is very difficult to construct irreducible free or nearly free curves.

Before we present the proofs, we need to recall very useful tools that we are going to use in our note. We start with the following crucial result \cite[Theorem 2.1]{DimcaSernesi}.
\begin{theorem}[Dimca-Sernesi]
\label{sern}
Let $C \, = \, \{ f=0 \}$ be a reduced curve of degree $d$ in $\mathbb{P}^{2}_{\mathbb{C}}$ having only quasi-homogeneous singularities. Then $${\rm mdr}(f) \geq \alpha_{C}\cdot d - 2,$$
where $\alpha_{C}$ denotes the Arnold exponent of $C$.
\end{theorem}
The Arnold exponent of a reduced curve $C \subset \mathbb{P}^{2}_{\mathbb{C}}$ is defined as the minimum over all log canonical thresholds ${\rm lct}_{p}(C)$ for $p \in {\rm Sing}(C)$. In the case when our singularities are just ordinary, we have the following result \cite[Theorem 1.3]{Cheltsov}. 
\begin{theorem}
Let $C$ be a reduced curve in $\mathbb{C}^{2}$ which has degree $m$ and let $p \in {\rm Sing}(C)$. Then ${\rm lct}_{p}(C) \geq \frac{2}{m}$, and the equality holds if and only if $C$ is a union of $m$ lines passing through $p$.
\end{theorem}
\noindent
By the above result, if $p \in \mathbb{C}^{2}$ is an ordinary singularity of multiplicity $r$ of $C$, then 
\begin{equation}
{\rm lct}_{p}(C) = \frac{2}{r}.
\end{equation}
Furthermore, if $q \in C$ is an ordinary cusp, then by \cite[Example 1.5]{Cheltsov} we have
\begin{equation}
{\rm lct}_{q}(C) = \frac{5}{6}.  
\end{equation}
Now we are ready to present our proof of \textbf{Theorem A}.
\begin{proof}
 The existence of nodal curves is granted by a result due to Severi \cite{Severi}. Since all singular points $p \in {\rm Sing}(C_{d})$ are nodes, we have ${\rm lct}_{p}(C_{d}) =1$, and the Arnold exponent of $C_{d}$ is equal to
 $$\alpha_{C_{d}} = 1.$$
 Then by Theorem \ref{sern} we have
 $${\rm mdr}(C_{d}) \geq d-2.$$
 By the assumption $d\geq 4$, so the following inequality holds
 $$d - 2 \geq \frac{d-2}{2},$$
 which means that by Theorem \ref{deff} the defect of $C_{d}$ is equal to 
 $$\nu(C_{d}) = \bigg\lceil\frac{3}{4}(d-1)^{2}\bigg\rceil - \tau(C_{d}).$$
 Now we want to find an upper bound on $\tau(C_{d})$. First of all, since all singularities of $C_{d}$ are nodes, one has $\tau_{p}(C_{d}) = 1$ for every $p \in {\rm Sing}(C_{d})$. Since the number of nodes of $C_{d}$ is bounded from above by $\frac{(d-1)(d-2)}{2}$, we get
 $$\tau(C_{d}) \leq \frac{(d-1)(d-2)}{2}.$$
 Taking into account the above inequality, we finally get 
 $$\nu(C_{d}) \geq \frac{3}{4}(d-1)^{2} - \frac{(d-1)(d-2)}{2} = \frac{1}{4}(d^{2}-1),$$
 which completes the proof.
\end{proof}
Now we pass to our proof of \textbf{Theorem B}.
\begin{proof}
The existence of such irreducible curves $K_{3k}$ of genus zero with $n_{3} = 2k$ ordinary triple points and nodes as the only singularities of $K_{3k}$ is granted by \cite[3.4 Theorem]{GM}. The condition that $K_{3k}$ has genus zero means that the curve has exactly 
$$n_{2} = \frac{9k^2 - 21k + 2}{2}$$
nodes as singularities.
Since curve $K_{3k}$ admits only nodes and ordinary triple points as singularities we get
$$\alpha_{K_{3k}} = {\rm min} \bigg\{1, \frac{2}{3}\bigg\} = \frac{2}{3},$$
 and then by Theorem \ref{sern}
$${\rm mdr}(K_{3k}) \geq \frac{2}{3}\cdot 3k-2 = 2k-2.$$
Since $k\geq 3$, we have
$$2k-2 > \frac{3k-2}{2},$$
so the defect of $K_{3k}$, by using Theorem \ref{deff}, can be bound from below
$$\nu(K_{3k}) =  \bigg\lceil\frac{3}{4}(3k-1)^{2}\bigg\rceil - 4 \cdot 2k - \frac{9k^2 - 21k + 2}{2} \geq \frac{1}{4}(9k+1)(k-1),$$
which completes the proof.
\end{proof}
Finally, we present our proof of \textbf{Theorem C}.
\begin{proof}
 We start by showing the existence of curves $C_{6k}$ with $k \geq 1$. In his Diplomarbeit, Ivinskis shows that there exists an irreducible and reduced curve $C_{6k}$ of degree $6k$ with $k\geq 1$ having exactly $9k^{2}$ ordinary cusps \cite[Lemma~4.1.7]{Ivinskis}. This curve is constructed using the Kummer cover $\kappa: \mathbb{P}^{2}_{\mathbb{C}} \ni (x,y,z) \mapsto (x^{k},y^{k},z^{k}) \in \mathbb{P}^{2}_{\mathbb{C}}$ applied to an irreducible and reduced sextic with exactly $9$ ordinary cusps. Recall that such an irreducible sextic is the dual curve to a smooth elliptic curve $E$, and the ordinary cusps correspond to the $9$ inflection points of $E$.

 Since our curve $C_{6k}$ admits only ordinary cusps as singularities,
 $$\alpha_{C_{6k}} = \frac{5}{6}$$
 and by Theorem \ref{sern} we have
 $${\rm mdr}(C_{6k}) \geq \frac{5}{6}\cdot 6k - 2 = 5k-2.$$
 Since for $k\geq 1$ one has
 $$5k-2 > \frac{6k-2}{2} = 3k - 1,$$
 and $\tau(C_{6k}) = 2\cdot 9k^{2} = 18k^{2}$, by Theorem \ref{deff}
 the defect of $C_{6k}$ is equal to
 $$\nu(C_{6k}) =  \bigg\lceil\frac{3}{4}(6k-1)^{2}\bigg\rceil - 18k^{2}.$$
 Observe that
 $$\bigg\lceil\frac{3}{4}(6k-1)^{2}\bigg\rceil = \bigg\lceil 27k^2 -9k +\frac{3}{4}\bigg\rceil = 27k^{2}-9k+1,$$
 and then
 $$\nu(C_{6k}) = 27k^{2}-9k+1 - 18k^2 = 9k^{2}-9k+1.$$
 In particular, for $k=1$ our curve $C_{6k}$ is an irreducible sextic with $9$ ordinary cusps with $\nu(C_{6})=1$, so $C_{6}$ is nearly free.
\end{proof}
\begin{remark}
Our curves $C_{6k}$ considered above are obviously not rational since
$$g(C_{6k}) = 9k^2 - 9k+1 \geq 1.$$
Moreover, it shows that $g(C_{6k}) = \nu(C_{6k})$, and this is very surprising that these two values coincide.
\end{remark}
Let us now present the main result of the note. Our result is devoted to reduced simply singular plane curves, i.e., reduced plane curves with only ${\rm ADE}$ singularities.

\begin{theoremD}[Non-freeness criterion]
\label{nfc}
Let $C \subset \mathbb{P}^{2}_{\mathbb{C}}$ be a reduced plane curve of even degree $d=2m\geq 4$ admitting only ${\rm ADE}$ singularities. Assume furthermore that the Arnold exponent of $C$ satisfies $\alpha_{C} \geq \frac{1}{2} + \frac{1}{m}$. Then
$$\nu(C) \geq 1.$$
In particular, $C$ is not free.
\end{theoremD}
\begin{proof}
The condition that $\alpha_{C} \geq \frac{1}{2} + \frac{1}{m}$ ensures us that the defect of $C$ can be computed via the second formula in Theorem \ref{deff}, namely
 $$\nu(C) = \bigg\lceil\frac{3}{4}(2m-1)^{2}\bigg\rceil - \tau(C).$$
 Observe that $$\bigg\lceil\frac{3}{4}(2m-1)^{2}\bigg\rceil = 3m^{2}-3m+1,$$
 so the last thing that we need to estimate is $\tau(C)$. By Theorem \ref{sern} we get ${\rm mdr}(C) \geq m$, which follows from the fact that $\alpha_{C} \geq \frac{1}{2} + \frac{1}{m}$. Now, using a result due to Du Plessis and Wall in \cite[Theorem 3.2]{duP}, we see that
 $$\tau(C) \leq \tau_{{\rm max}}(2m,r) := (2m-1)(2m-r-1)+r^{2} -\binom{2r-2m+2}{2},$$
 where $r:={\rm mdr}(C)$.
 Since the function $\tau_{{\rm max}}(2m,r)$ is strictly decreasing as a function with respect to $r$ on the interval $[m, 2m-1]$, we get
 $$\tau(C) \leq  \tau_{{\rm max}}(2m,m) = 3m^{2}-3m,$$
 so we finally obtain
 $$\nu(C) = 3m^{2}-3m+1 -\tau(C) \geq 3m^{2}-3m+1 - \tau_{{\rm max}}(2m,m) = 1,$$
 which completes the proof.
\end{proof}
Now we present the following example to show that our main result is optimal.
\begin{example}
This example comes from \cite[7.5 Lemma]{Persson}. Fix an even integer $m \in \mathbb{Z}_{\geq 4}$ and consider the curve $\mathcal{C}_{2m} = \{C_{1}, C_{2}, C_{3}, C_{4}\} \subset \mathbb{P}^{2}_{\mathbb{C}}$, where
\begin{equation*}
\begin{array}{l}
C_{1} :\quad  x^{m/2} + y^{m/2} + z^{m/2} =0,\\
C_{2} :\, -x^{m/2} + y^{m/2} + z^{m/2} =0,\\
C_{3} :\quad  x^{m/2} - y^{m/2} + z^{m/2} =0,\\
C_{4} :\quad x^{m/2} + y^{m/2} - z^{m/2} =0.
\end{array}
\end{equation*}
Our curve $\mathcal{C}_{2m}$ is of degree $d=2m$ and it has $3m$ singularities of type $A_{m-1}$, see \cite[Lemma 7.5]{Persson}. In particular, for $m=4$ we obtain the arrangement of $4$ conics that admits exactly $12$ singularities of type $A_{3}$ -- it is well-known that this arrangement is unique up to the projective equivalence. 

Since $\mathcal{C}_{2m}$ admits only singularities of type $A_{m-1}$, for each $p \in {\rm Sing}(\mathcal{C}_{2m})$ one has ${\rm lct}_{p}=\frac{1}{2} + \frac{1}{m}$, so the Arnold exponent of $\mathcal{C}_{2m}$ is equal to 
$$\alpha_{\mathcal{C}_{2m}} = \frac{1}{2} + \frac{1}{m}.$$ By Theorem D, we have 
$$\nu(\mathcal{C}_{2m}) \geq 1.$$
In fact, based on \cite[Theorem 3.12]{Artal}, our curve $\mathcal{C}_{2m}$ is nearly free, i.e., $\nu(\mathcal{C}_{2m})=1$.
\end{example}
Finally, let us present an application of Theorem D in the setting of line arrangements.
\begin{example}
Consider an arrangement $\mathcal{L}$ of $d=2m$ lines, which only admits double and triple intersections. Then $\alpha_{\mathcal{L}} = \frac{2}{3}$, and if we assume that $m\geq 6$, then
$$\alpha_{\mathcal{L}} \geq \frac{1}{2} + \frac{1}{m}.$$
Using Theorem D we can conclude that there is no free arrangement of $d=2m\geq 12$ lines with double and triple intersections.
\end{example}

\section*{Acknowledgement}
I would like to thank Emilia Mezzetti for useful explanations regarding the content of \cite{GM}. I would also like to thank the anonymous referees for their valuable comments, which allowed me to improve the note.

Piotr Pokora is supported by the National Science Centre (Poland) Sonata Bis Grant  \textbf{2023/50/E/ST1/00025}. For the purpose of Open Access, the author has applied a CC-BY public copyright licence to any Author Accepted Manuscript (AAM) version arising from this submission.

\vskip 0.5 cm
\bigskip
Piotr Pokora,
Department of Mathematics,
University of the National Education Commission Krakow,
Podchor\c a\.zych 2,
PL-30-084 Krak\'ow, Poland. \\
\nopagebreak
\textit{E-mail address:} \texttt{piotr.pokora@up.krakow.pl}
\end{document}